\documentclass[review]{elsarticle}
\usepackage[utf8]{inputenc}
\usepackage[english]{babel}
\usepackage{a4wide}
\usepackage{amssymb}
\usepackage{amsmath}
\usepackage{amsthm}
\usepackage{setspace}
\usepackage{graphicx}
\usepackage{enumitem}
\usepackage{color}
\usepackage[top = 2cm, bottom = 2cm, inner = 2cm, outer = 2cm]{geometry}

\newtheorem{lemma}{Lemma}
\newtheorem{theorem}{Theorem}
\newtheorem*{theorem_old}{Theorem}
\newtheorem*{corollary_old}{Corollary}
\newtheorem{corollary}{Corollary}

\theoremstyle{definition}

\newtheorem{remark}{Remark}

\DeclareMathOperator{\trdeg}{trdeg}
\DeclareMathOperator{\wronsk}{wr}

\begin{document}

\title{The primitive element theorem for differential fields with zero derivation on the ground field}

\author[me]{Gleb A. Pogudin}
\ead{pogudin.gleb@gmail.com}
\address[me]{Department of Higher Algebra, Faculty of Mechanics and Mathematics, Moscow State University, Leninskie Gory 1, GSP-1, Moscow,  119991, Russia}

\begin{abstract}

    In this paper we strengthen Kolchin's theorem (\cite{kolchin}) in the ordinary case.
    It states that if a differential field $E$ is finitely generated over a differential subfield $F \subset E$, $\trdeg_F E < \infty$, and $F$ contains a nonconstant, i.e. an element $f$ such that $f^{\prime} \neq 0$, then there exists $a \in E$ such that $E$ is generated by $a$ and $F$.
    We replace the last condition with the existence of a nonconstant element in $E$.

\end{abstract}


\maketitle

\onehalfspacing

\section*{Introduction}

All fields considered in this paper are of characteristic zero.

Let $R$ be a ring.
A map $D\colon R \to R$ satisfying $D(a + b) = D(a) + D(b)$ and $D(ab) = aD(b) + D(a)b$ for all $a, b \in R$ is called \textit{derivation}.
We will denote $D(x)$ by $x^{\prime}$ and $D^n(x)$ by $x^{(n)}$.

A \textit{differential ring} $R$ is a ring with a specified derivation. 
A differential ring which is a field will be called a \textit{differential field}.
Let $F \subset E$ be a differential field extension and $a \in E$. 
Let us denote by $F\langle a \rangle$ the differential subfield of $E$ generated by $F$ and $a$.
If $F\langle a \rangle = E$, then element $a$ is said to be \textit{primitive}.

An element $a \in R$ of the differential ring $R$ is said to be \textit{constant} if $a^{\prime} = 0$.

%%%%%%%%%%%%%%%%%%%%%%%%%%

Kolchin proved (\cite{kolchin}) a differential analogue of the primitive element theorem:

\begin{theorem_old}
    Let $E = F\langle a_1, \ldots, a_n \rangle$ and $\trdeg_F E < \infty$. 
    Assume also that $F$ contains a nonconstant element. 
    Then, there exists $b \in E$ such that $E = F \langle b \rangle$.
\end{theorem_old}

\begin{corollary_old}
    Let $E = F\langle a_1, \ldots, a_n \rangle$ and $\trdeg_F E < \infty$. 
    Assume also that $E$ contains a nonconstant element. 
    Then, there exist $b, c \in E$ such that $E = F \langle b, c \rangle$. 
\end{corollary_old}

%\begin{theorem_old}
%    Let $E = F\langle a_1, \ldots, a_n \rangle$ and $\trdeg_F E < \infty$.
%    Assume that $E$ and $F$ are both differential fields with respect to the set of commuting derivations $\{ \Delta_1, \ldots, \Delta_m\}$.
%    Assume also that there exist $c_1, \ldots, c_m \in E$ such that $\det\left( (\Delta_i c_j)_{i,j=1}^{m} \right) \neq 0$.
%    Then there exists such $b \in E$ that $E = F\langle b \rangle$.
%\end{theorem_old}

\begin{remark}
    In \cite{kolchin} Kolchin considered a more general case, i.e. fields equipped with a set of commuting derivations. We restrict ourselves to the ordinary case.
\end{remark}

%%%%%%%%%%%%%%%%%%%%%%%%%%

In \cite{bab} Babakhanian constructed primitive elements for several specific extensions $F \subset E$ with $F$ consisting of constant elements.

The goal of the present paper is to prove the primitive element theorem for the case $f^{\prime} = 0$ for all $f \in F$.

%%%%%%%%%%%%%%%%%%%%%%%%%%%%%%%%%%%%%%%%%%%%%%%%%%%%%%%%%%%%%%%%%%%%%%%%%%%%%%%%%%%%%%%

\section*{Main results}

\begin{theorem}\label{th:density}
    Let $E = k \langle a, b \rangle$, $\trdeg_k E < \infty$, and $b^{\prime} \neq 0$. 
    Then, there exists $p(x) \in k[x]$ such that $\trdeg_k k\langle a + p(b) \rangle = \trdeg_k k\langle a, b \rangle$.
\end{theorem}

\begin{theorem}\label{th:primitive}
    Let $E = k\langle a_1, \ldots, a_m \rangle$, $\trdeg_k E < \infty$, and $E$ contains a nonconstant. 
    Then, there exists $a \in E$ such that $E = k\langle a \rangle$.
\end{theorem}

\begin{remark}
    Unlike Kolchin's proof it is not sufficient to consider elements of the form $a + \lambda b$ ($\lambda \in k$). 
    For example, let $\mathbb{Q}(x, y)$ be a differential field with the derivation defined by $x^{\prime} = 1$ and $y^{\prime} = 0$. 
    There is no primitive element of the form $y + \lambda x$ ($\lambda \in \mathbb{Q}$), but $\mathbb{Q}(x, y) = \mathbb{Q} \langle x^2 + y \rangle$.
\end{remark}

%%%%%%%%%%%%%%%%%%%%%%%%%%%%%%%%%%%%%%%%%%%%%%%%%%%%%%%%%%%%%%%%%%%%%%%%%%%%%%%%%%%%%%%

\begin{proof}[Proof of Theorem \ref{th:density}]

    We will need the following well-known lemmas:
    \begin{lemma}
        If $\trdeg_k k\langle a \rangle = n$, then $k\langle a \rangle = k\left( a, a^{\prime}, \ldots, a^{(n)} \right)$.
    \end{lemma}
    \begin{proof}
        Let $m$ be the minimal integer such that $a, \ldots, a^{(m)}$ are algebraically dependent over $k$. 
        Let $R(a, \ldots, a^{(m)}) = 0$ be the corresponding algebraic relation. 
        Hence 
        
        $$0 = \left( R(a, \ldots, a^{(m)}) \right)^{\prime} = \sum\limits_{i = 0}^{m} a^{(i + 1)} \frac{\partial}{\partial a^{(i)}} R$$
        
        , so $a^{(m + 1)} \in k(a, \ldots, a^{(m)})$.

        Similarly we obtain that $a^{(N)} \in k(a, \ldots, a^{(m)})$ for all $N$. 
        Hence, $n = m$ and $k\langle a\rangle = k\left( a, \ldots, a^{(n)} \right)$.
    \end{proof}

    \begin{lemma}[\cite{ritt}, p.35]\label{lem:nonzero}
        Let $q(x, x^{\prime}, \ldots, x^{(n)})$ be a nonzero differential polynomial over a differential field $E$. 
        Let $f \in E$ be a noncostant element. 
        Then, there exists $p(t) \in \mathbb{Q}[t]$ such that $$\left.q(x, x^{\prime}, \ldots, x^{(n)})\right\rvert_{x = p(f)} \neq 0$$
    \end{lemma}

    Without loss of generality, we can assume that $E = k\langle a, b\rangle$. 
    Let us introduce algebraically independent variables $\Lambda_0, \Lambda_1, \ldots$. 
    We extend the derivation from $E$ to $E[\Lambda_0, \Lambda_1, \ldots]$ by $\left( \Lambda_i \right)^{\prime} = b^{\prime} \Lambda_{i + 1}$. 
    This construction can be explained by the following observation: let us fix a polynomial $p(x) \in \mathbb{Q}[x]$; the formulas $\varphi_p(\Lambda_i) = p^{(i)}(b)$ define a homomorphism of differential $k$-algebras $\varphi_p\colon E[\Lambda_0, \Lambda_1, \ldots] \to E$.

    Let $c = a + \Lambda_0$ and $K = k(\Lambda_0, \Lambda_1, \ldots) \subset E(\Lambda_0, \Lambda_1, \ldots)$. 
    Since $K\langle c \rangle \subset K\langle a , b \rangle$, $\trdeg_K K\langle c \rangle = n < \infty$. 
    Let nonzero $R(x_0, \ldots, x_n) \in K[x_0, \ldots, x_n]$ satisfy $R(c, c^{\prime}, \ldots, c^{(n)}) = 0$. 
    Notice that it depends on $x_n$. 
    Multiplying by the suitable element of $k[\Lambda_0, \Lambda_1, \ldots]$, we obtain a polynomial in both $c, c^{\prime}, \ldots, c^{(n)}$ and $\Lambda_0, \ldots, \Lambda_N$ over $k$. 
    Let us denote it by $Q(c, \ldots, c^{(n)}, \Lambda_0, \ldots, \Lambda_N)$. 
    Moreover, we assume that $Q$ satisfies the following conditions:

    \begin{enumerate}
        \item $\deg_{c^{(n)}} Q$ is a minimal possible;

        \item under the above condition, $N$ is a minimal possible;

        \item under the above conditions, $\deg_{\Lambda_N} Q$ is minimal possible.
    \end{enumerate}

    \begin{lemma}
        $N = n$.
    \end{lemma}
    \begin{proof}
        Assume that $N > n$.
        Let us rewrite $Q$ as a polynomial in $\Lambda_N$: $Q = q_m \Lambda_N^m + \ldots + q_0$, where $q_i$ are polynomials over $k$ in $c, \ldots, c^{(n)}, \Lambda_0, \ldots, \Lambda_{N - 1}$. 
        Since $N > n$, $c, \ldots, c^{(n)} \in E(\Lambda_0, \ldots, \Lambda_{N - 1})$ and $\Lambda_N$ is transcendental over $k(c, \ldots, c^{(n)}, \Lambda_0, \ldots, \Lambda_{N - 1})$. 
        Thus, $Q = 0$ implies $q_i = 0$ for all $i$. 
        We obtained a contradiction with minimality of $N$. 

        Assume that $N < n$. 
        Clearly, $c^{(n)} = (b^{\prime})^n \Lambda_n + c_0$, where $c_0 \in E(\Lambda_0, \ldots, \Lambda_{n - 1})$. 
        Thus, $c^{(n)}$ is transcendental over $k(c, \ldots, c^{(n - 1)}, \Lambda_0, \ldots, \Lambda_N) \subset E(\Lambda_0, \ldots, \Lambda_{n - 1})$. 
        But $Q$ depends on $c^{(n)}$. 
        This contradiction proves the lemma.
    \end{proof}

    \begin{lemma}\label{lem:part_dev_nonzero}
        $\frac{\partial}{\partial\Lambda_n} Q \neq 0$.
    \end{lemma}
    \begin{proof}
        It follows immediately from the minimality conditions for $Q$ and inequalities $\deg_{c^{(n)}} Q \geqslant \deg_{c^{(n)}} \frac{\partial}{\partial \Lambda_n} Q$ and $\deg_{\Lambda_n} Q > \deg_{\Lambda_n} \frac{\partial}{\partial\Lambda_n} Q > -\infty$.
    \end{proof}

    We return to the proof of Theorem~\ref{th:density}.

    Let $p(x) \in \mathbb{Q}[x]$. 
    Applying $\varphi_p$ to $Q(c, \ldots, c^{(n)}, \Lambda_0, \ldots, \Lambda_n)$, we obtain an algebraic dependence for $\varphi_p(\Lambda_0), \ldots, \varphi_p(\Lambda_n) \in \mathbb{Q}[b]$ over $k\left( \varphi_p(c), \ldots, \varphi_p (c^{(n)})\right)$. It yields to an algebraic dependence for $b$ over $k\left( \varphi_p(c), \ldots, \varphi_p(c^{(n)}) \right)$.
    The goal is to find an appropriate $p(x)$ to make this dependence nontrivial. 
    Let us compute its derivation with respect to $b$:

    \begin{multline} 
\frac{\partial}{\partial b} Q\left( \varphi_p(c), \ldots, \varphi_p(c^{(n)}), p(b), \ldots, p^{(n)}(b) \right) = \\ =\quad{} \sum\limits_{i = 0}^{n} \varphi_p \left( \frac{\partial}{\partial \Lambda_i} Q\right) p^{(i + 1)}(b) \quad{} = \quad \varphi_p \left( \sum\limits_{i = 0}^{n} \Lambda_{i + 1} \frac{\partial}{\partial \Lambda_i} Q \right) \end{multline}

    By Lemma \ref{lem:part_dev_nonzero} the polynomial $T = \sum\limits_{i = 0}^{n} \Lambda_{i + 1} \frac{\partial}{\partial\Lambda_i} Q$ is nonzero. 
    Since $c = a + \Lambda_0$, we can rewrite $T$ as a nonzero polynomial in $\Lambda_0, \ldots, \Lambda_{n + 1}$ over $k\langle a, b\rangle$. 
    Let us denote the derivation on $E$ by $D$. 
    Then, $\tilde{D} = \frac{1}{b^{\prime}} D$ is also a derivation on $E$. 
    Obviously, $\tilde{D} \Lambda_i = \Lambda_{i + 1}$.

    Hence, we can apply Lemma \ref{lem:nonzero} to the differential field $E$ with respect to $\tilde{D}$, nonconstant element $b \in E$, and the polynomial $T$ in variables $\Lambda_0, \ldots, \Lambda_n$. 
    Therefore, we obtain such $p(x)\in \mathbb{Q}[x]$ that $\varphi_p (T) \neq 0$. 

    Since $\varphi_p(c) = a + p(b)$ and $b$ are both algebraic over $k\langle \varphi_p(c) \rangle$, $a$ is also algebraic over $k\langle \varphi_p(c) \rangle$. 
    Hence, $\trdeg_k k\langle \varphi_p(c) \rangle = \trdeg k\langle a, b \rangle$.

\end{proof}

The following corollary can be derived using exactly the same argument as above. 

\begin{corollary}\label{cor:density}
    Let $E = k\langle a, b \rangle$, $\trdeg_k E < \infty$, $b^{\prime} \neq 0$ and $c\in k\langle b\rangle$. 
    Then, there exists $p(x) \in k[x]$ such that $\trdeg k\langle a + c \cdot p(b) \rangle = \trdeg k\langle a, b \rangle$.
\end{corollary}

%%%%%%%%%%%%%%%%%%%%%%%%%%%%%%%%%%%%%%%%%%%%%%%%%%%%%%%%%%%%%%%%%%%%%%%%%%%%%%%%%%%%%%%%%%

\begin{proof}[Proof of Theorem \ref{th:primitive}.]

Due to Theorem \ref{th:density} there exists $a\in E$ such that $\trdeg_k E = \trdeg_k k\langle a \rangle = n$. 
Since $\dim_{k\langle a \rangle} E < \infty$ there exists $b\in E$ such that $E = k\langle a, b\rangle$. 
We are going to find $\lambda_1, \ldots, \lambda_{n + 2} \in k$ such that $E = k\langle b + \lambda_1 a + \lambda_2 a^2 + \ldots + \lambda_{n + 2} a^{n + 2}\rangle$.

%Moreover:

%\begin{lemma}
%    Let $\trdeg_k k\langle b \rangle = n$ and $b^{\prime} \neq 0$. 
%    Then there exists such $p(t) \in \mathbb{Q}[t]$ that $c = p(b)$ is algebraic over $k\left( c^{\prime}, \ldots, c^{(n)}\right)$.
%\end{lemma}

%\begin{proof}
%    Let us apply the corollary \ref{cor:density} assuming $a = 0$, $b = b$ and $c = b^{\prime}$.
%    Henceworth (???), we obtain the polynomial $q(t) \in \mathbb{Q}[t]$ such that $b$ is algebraic over $\tilde{c} = b^{\prime} q(b), \tilde{c}^{\prime}, \ldots, \tilde{c}^{(n - 1)}$. 
%    Let $p(t) = \int q(t) \dt$ and $c = p(b)$. 
%    Then $\tilde{c} = c^{\prime}$ and $b$ is algebraic over $k\left( c^{\prime}, \ldots, c^{(n)} \right)$. 
%    Since $c$ is algebraic over $k(b)$, the proof of the lemma is complete.
%\end{proof}

%%%%%%%%%%%%%%%%%%%%%%%%%%%%%%%%%%%%%%%%%%%%%%%%%%%%%%%%%%%%%%%%%%%%%%%%%%%%

%Thus, we can also assume that $b$ is algebraic over $k\left( b^{\prime}, \ldots, b^{(n)} \right)$. 

We will use the method used by Kolchin in \cite[p.729]{kolchin}.
Let us recall necessary definitions.

Let $K_1$ be a differential extension field of $L$. By an \textit{isomorphism  of $K_1$ with respect to $L$} we will mean an isomorphic mapping of $K_1$ onto a differential field $K_2$ such that
\begin{enumerate}
    \item $K_2$ is an extension of $L$;
    \item the isomorphic mapping leaves each element of $L$ invariant;
%    \item $K_1$ and $K_2$ have a common extension.
\end{enumerate}

\begin{lemma}[Kolchin, (\cite{kolchin}, p.726)] \label{lem:invariant}
    Let $E$ be an extension of $F$ and $\gamma \in E$. 
    A necessary and sufficent condition for $E = F\langle \gamma \rangle$ is that no isomorphism of $E$ with respect to $F$ other then the identity leaves $\gamma$ invariant.
\end{lemma}

%%%%%%%%%%%%%%%%%%%%%%%%%%%%%%%%%%%%%%%%%%%%%%%%%%%%%%%%%%%%%%%%%%%%%%%%%%%%%%

{Let $I_0 = \{f(x, y) \in k\{x, y\} \mid f(a, b) = 0\}$ be the ideal of differential polynomials vanishing at $(a, b)$.} 
%Let $R(x, x^{\prime}, \ldots, x^{(n)}) \in k\{ x \}$ (shortly, $R(x)$) have a solution $x = a$ and $Q(x, x^{\prime}, \ldots, x^{(n - 1)}, y)\in k\{x, y\}$ (shortly, $Q(x, y)$) have a solution $x = a$, $y = b$. 
We will show that there exist elements $\lambda_1, \ldots, \lambda_n$ such that $z = y + \lambda_1 x + \ldots + \lambda_n x^n$ takes different values for different {generic} solutions of $I_0$ {(see \cite[p.725]{kolchin})}. 
Then, certainly $z$ will satisfy the requirements on $\gamma$ from Lemma \ref{lem:invariant}.

To prove this statement, let $t_1, \ldots, t_{n + 2}$ be new indeterminates, and, in $E\{x, y, t_1, \ldots, t_{n + 2}\}$, consider the perfect differential ideal (for definitions see \cite[p.2, p.7]{ritt}) 

$$I = \{ {I_0}, t_1^{\prime}, \ldots, t_{n + 2}^{\prime}, b - y + t_1(a - x) + t_2(a^2 - x^2) + \ldots + t_{n + 2}(a^{n + 2} - x^{n + 2}) \}$$

Let $I = \mathfrak{p}_1 \cap \ldots \cap \mathfrak{p}_s$ be the decomposition of $I$ into essential prime differential ideals (see \cite[p.13]{ritt}), and suppose the subscripts have been assigned so that {for every $1 \leqslant i \leqslant r$ one of the following conditions hold
\begin{enumerate}[label=(C\arabic*)]
    \item\label{cond:same} $a - x, b - y \in \mathfrak{p}_i$;
    \item\label{cond:nongeneric} $k\{x, y\}\cap \mathfrak{p}_i \neq I_0$.
\end{enumerate}
Consider $\mathfrak{p}_i$ with $i \leqslant r$ and its solution $(\overline{x}, \overline{y}, \overline{t}_1, \ldots, \overline{t}_{n + 2})$.
If~\ref{cond:same} holds, then $(a, b) = (\overline{x}, \overline{y})$.
If~\ref{cond:nongeneric} holds, $(\overline{x}, \overline{y})$ is not a generic solution of $I_0$.
}
Consider $\mathfrak{p}_j$ with $j > r$. 
If $b - y \notin \mathfrak{p}_j$, then also $a - x \notin \mathfrak{p}_j$. 
Thus, $a - x \notin \mathfrak{p}_j$ for all $j > r$. 
Let $\overline{x}, \overline{y}, \overline{t}_1, \ldots, \overline{t}_n$ be a generic solution of $\mathfrak{p}_j$.
{Then $(\overline{x}, \overline{y})$ is a generic solution of $I_0$.
In particular, $\overline{y}$ is algebraic over $k\langle \overline{x} \rangle$.}
Differentiating the equation 

$$
b -\overline{y} + \overline{t}_1(a - \overline{x}) + \overline{t}_2(a^2 - \overline{x}^2) + \ldots + \overline{t}_n(a^n - \overline{x}^n) = 0
$$
$n + 1$ times, we obtain a system of linear equations in $\overline{t}_1, \ldots, \overline{t}_{n + 2}$. Let us investigate it.

Let us denote by $\wronsk(f_1, \ldots, f_N)$ the wronskian of $f_1$, $\ldots$, $f_N$ (see \cite[chap. 2]{magid}). 

%\begin{lemma}\label{lem:polyn_diff}
%    Let $p(x, y) \in k[x, y]$ be a polynomial. If for every $t \in k$ holds $p(x + t, y + t) = p(x, y)$, then there exists such $q(z) \in k[z]$ that $p(x, y) = q(x - y)$.
%\end{lemma}

%\begin{proof}
%    Since $x = (x - y) + y$, $p(x, y)$ can be considered as an element of $k[x - y, y]$ and rewritten as $c_N(y)(x - y)^N + \ldots + c_0(y)$. For a polynomial $c(y) \in k[y]$ equality $c(y) = c(y + t)$ with $t \neq 0$ implies $c(y) \in k$. Thus, $c_N(y), \ldots, c_0(y) \in k$.
%\end{proof}

%%%%%%%%%%%%%%%%%%%%%%%%%%%%%%%%%%%%%%%%%%%%%%%%%%%%%%%%%%%%%%

%\begin{lemma}
%    Let $x, y$ be differential indeteminates. Then 
%    $$\wronsk(x - y, x^2 - y^2, \ldots, x^k - y^k) \in \mathbb{Q}[x^{\prime}, y^{\prime}, \ldots, x^{(k - 1)}, y^{(k - 1)}, x - y]$$
%\end{lemma}

%\begin{proof}
%    All we need is to verify the condition of the lemma \ref{lem:polyn_diff}.
%    Since $(x + t)^m - (y + t)^m = \sum\limits_{i = 0}^m C_m^i t^i (x^{m - i} - y^{m - i})$, then:
%
%    $$ \wronsk(x - y, \ldots, x^k - y^k) = \wronsk((x + t) - (y + t), \ldots, (x + t)^k - (y + t)^k) $$
%\end{proof}

%%%%%%%%%%%%%%%%%%%%%%%%%%%%%%%%%%%%%%%%%%%%%%%%%%%%%%%%%%%%%%%

Let $W_{k, l}(x, y)$ be given by $\wronsk(x - y, \ldots, \widehat{x^l - y^l}, \ldots, x^{k + 1} - y^{k + 1})$ where $k \geqslant 2$ and $1 \leqslant l \leqslant k + 1$.

\begin{lemma}\label{lem:trdeg_estim1}
    If $W_{k, l}(a, \overline{x}) = 0$ for all $1 \leqslant l \leqslant k + 1$, then $\trdeg_k k\langle a, \overline{x} \rangle \leqslant n + k - 2$.
\end{lemma}

\begin{proof}

    Let $x$ and $y$ be differential indeterminates. 
    First of all, we are going to establish several properties of differential polynomials $W_{k, l}(x, y)$.

    \begin{lemma}
        $W_{k, l}(x, y) = A_l(x, y) + x^{(k - 1)}B_l(x, y) + y^{(k - 1)}C_l(x, y)$ where $A_l, B_l, C_l \in \mathbb{Q}[x, \ldots, x^{(k - 2)}, y, \ldots, y^{(k - 2)}]$. 
        Moreover, if $k \geqslant 3$, then $B_l(x, y) = -y^{\prime}D_l(x, y)$ and $C_l(x, y) = x^{\prime}D_l(x, y)$ where $D_l \in \mathbb{Q}[x, \ldots, x^{(k - 2)}, y, \ldots, y^{(k - 2)}]$.
    \end{lemma}

    \begin{proof}
        For the sake of simplicity let us consider $l = k + 1$.
        The proof for the other cases is analogous.
        The last row of the corresponding matrix is a sum of three rows: $x^{(k - 1)}(1, 2x, \ldots, kx^{k - 1})$, $-y^{(k - 1)}(1, 2y, \ldots, y^{k - 1})$ and $(a_1, \ldots, a_k)$ where $a_i \in \mathbb{Q}[x, \ldots, x^{(k - 2)}, y, \ldots, y^{(k - 2)}]$ for all $i$.
        Thus, the determinant $W_{k, k + 1}(x, y)$ can be expressed as a sum:

        \resizebox{\hsize}{!}{$
        \begin{vmatrix}
            x - y                     & \ldots & x^k - y^k                        \\
            \vdots                    & \ddots & \vdots                           \\
            (x - y)^{(k - 2)}         & \ldots & \left(x^k - y^k\right)^{(k - 2)} \\
            a_1                       & \ldots & a_k
        \end{vmatrix}
        +
        x^{(k - 1)}
        \begin{vmatrix}
            x - y                     & \ldots & x^k - y^k                        \\
            \vdots                    & \ddots & \vdots                           \\
            (x - y)^{(k - 2)}         & \ldots & \left(x^k - y^k\right)^{(k - 2)} \\
            1                         & \ldots & kx^{k - 1}
        \end{vmatrix}
        -
        y^{(k - 1)}
        \begin{vmatrix}
            x - y                     & \ldots & x^k - y^k                        \\
            \vdots                    & \ddots & \vdots                           \\
            (x - y)^{(k - 2)}         & \ldots & \left(x^k - y^k\right)^{(k - 2)} \\
            1                         & \ldots & ky^{k - 1}
        \end{vmatrix}
        $}

        The above equality proves the first statement of the lemma.

        Now let $k \geqslant 3$.
        The second row of the corresponding matrix is a sum of $x^{\prime}(1, 2x, \ldots, kx^{k - 1})$ and $-y^{\prime}(1, 2y, \ldots, ky^{k - 1})$. Hence, substracting the last row from the second, we obtain:

        $$
        B_l = 
        \begin{vmatrix}
            x - y                   & \ldots & x^k - y^k                                   \\
            x^{\prime} - y^{\prime} & \ldots & x^{\prime}kx^{k - 1} - y^{\prime}ky^{k - 1} \\
            \vdots                  & \ddots & \vdots                                      \\
            1                       & \ldots & kx^{k - 1}
        \end{vmatrix}
        =
        -y^{\prime}
        \begin{vmatrix}
            x - y                   & \ldots & x^k - y^k                                   \\
            1                       & \ldots & ky^{k - 1}                                  \\
            \vdots                  & \ddots & \vdots                                      \\
            1                       & \ldots & kx^{k - 1}
        \end{vmatrix}
        $$

        Let us denote the latter determinant by $D_l$.
        Then, $B_l = -y^{\prime}D_l$.
        Likewise, $C_l = x^{\prime}D_l$, so we are done.
    \end{proof}

    %%%%%%%%%%%%%%%%%%%%%%%%%%%%%%

    \begin{lemma}
        \label{lem:cramer}
        At least one of $\frac{W_{k, 1}(x, y)}{W_{k, k + 1}(x, y)}$, $\ldots$, $\frac{W_{k, k}(x, y)}{W_{k, k + 1}(x, y)}$ depends on $x^{(k - 1)}$ and $y^{(k - 1)}$.
    \end{lemma}
    \begin{proof}
        Since all these differential polynomials are symmetric in $x$ and $y$, it suffices to prove that at least one of them depends on either $x^{(k - 1)}$ or $y^{(k - 1)}$.
        Assume the contrary, that $(-1)^{k - l}\frac{W_{k, l}(x, y)}{W_{k, k + 1}(x, y)} \in \mathbb{Q}(x, \ldots, x^{(k - 2)}, y, \ldots, y^{(k - 2)})$.
        By the Cramer's rule these fractions are solutions of the following system of linear equations in $\alpha_1, \ldots, \alpha_k$:

        \begin{equation}
        \label{eq:system}
        \tag{*}
        \begin{pmatrix}
            x - y                     & \ldots & x^k - y^k                          \\
            \vdots                    & \ddots & \vdots                             \\
            (x - y)^{(k - 1)} & \ldots & \left( x^k - y^k \right)^{(k - 1)}
        \end{pmatrix}
        \begin{pmatrix}
            \alpha_1 \\
            \vdots   \\
            \alpha_k
        \end{pmatrix}
        =
        \begin{pmatrix}
            x^{k + 1} - y^{k + 1}                          \\
            \vdots                                         \\
            \left( x^{k + 1} - y^{k + 1} \right)^{(k - 1)}
        \end{pmatrix}
        \end{equation}

        Since $\alpha_1, \ldots, \alpha_k \in \mathbb{Q}(x, \ldots, x^{(k - 2)}, y, \ldots, y^{(k - 2)})$ and both $x^{(k - 1)}$ and $y^{(k - 1)}$ are transcendental over this field, the last equation also implies two following equalities:

        \begin{equation}
        \label{eq:cons}
        \tag{**}
        \begin{cases}
            \alpha_1 + 2x\alpha_2 + \ldots + kx^{k - 1}\alpha_k = (k + 1)x^{k} \\
            \alpha_1 + 2y\alpha_2 + \ldots + ky^{k - 1}\alpha_k = (k + 1)y^{k}
        \end{cases}
        \end{equation}

        We are going to assign values from a particular differential field to $x$ and $y$. 
        Precisely, let $\mathbb{C}(t)$ be a field of rational functions equipped with standard derivation ($t^{\prime} = 1$) and $\xi$ be a primitive $(k + 1)$-th root of unity. 
        Let $x = t$ and $y = \xi t$. 
        The matrix of the system (\ref{eq:system}) is nondegenerate, because its determinant equals $\wronsk\left( (1 - \xi)t, \ldots, (1 - \xi^k)t^k \right)$, which is nonzero since $(1 - \xi)t$, $\ldots$, $(1 - \xi^k)t^k$ are linearly independent over constants (\cite[prop. 2.8]{magid}). 
        Clearly, the unique solution of the system (\ref{eq:system}) in this case is $\alpha_1 = \ldots = \alpha_k = 0$. 
        But the equalities (\ref{eq:cons}) do not hold. 
        This contradiction proves the lemma.
    \end{proof}

    \begin{corollary}
        \label{cor:alg_relation}
        If $k \geqslant 3$ then there exists $1 \leqslant l \leqslant k$ such that $W_{k, l}(x, y)D_{k + 1}(x, y) - W_{k, k + 1}(x, y)D_l(x, y) = A_l(x, y)D_{k + 1}(x, y) - A_{k + 1}(x, y)D_l(x, y) \neq 0$.
    \end{corollary}
    \begin{proof}
        By the Lemma \ref{lem:cramer} there exists $1 \leqslant l \leqslant k$ such that $\frac{W_{k, l}(x, y)}{W_{k, k + 1}(x, y)}$ depends on $x^{(k - 1)}$ and $y^{(k - 1)}$. 
        This means that vectors $(A_l, B_l, C_l)$ and $(A_{k + 1}, B_{k + 1}, C_{k + 1})$ are not proportional. 
        Thus, $D_{k + 1}(A_l, B_l, C_l) - D_l(A_{k + 1}, B_{k + 1}, C_{k + 1}) = (A_lD_{k + 1} - A_{k + 1}D_l, 0, 0) \neq 0$.
    \end{proof}

    We return to the proof of Lemma~\ref{lem:trdeg_estim1}.

    Let us consider two cases:
    \begin{enumerate}
        \item $k \leqslant 3$. 
            Let $l$ be the index from the corollary \ref{cor:alg_relation}. 
            Since $W_{k, l}(a, \overline{x}) = W_{k, k + 1}(a, \overline{x}) = 0$, $W_{k, l}(a, \overline{x})D_{k + 1}(a, \overline{x}) - W_{k, k + 1}(a, \overline{x})D_l(a, \overline{x})$ provides us an algebraic relation between $a$, $\ldots$, $a^{(k - 2)}$, $\overline{x}$, $\ldots$, $\overline{x}^{(k - 2)}$. 
            Hence, $\overline{x}^{(j)}$ is algebraic over $k(a, \ldots, a^{(n - 1)}, \overline{x}, \ldots, \overline{x}^{(k - 3)})$ for all $j$.
            Thus, $k\langle a, \overline{x} \rangle$ is algebraic over $k(a, \ldots, a^{(n - 1)}, \overline{x}, \ldots, \overline{x}^{(k - 3)})$.
            Since $\trdeg_k k(a, \ldots, a^{(n - 1)}, \overline{x}, \ldots, \overline{x}^{(k - 3)}) \leqslant n + k - 2$, we are done.

        \item $k = 2$.
            In this case both $W_{2, 2}(x, y)$ and $W_{2, 3}(x, y)$ can be computed directly:

            $$W_{2, 3}(x, y) = (x - y)^2(x^{\prime} + y^{\prime})$$

            $$W_{2, 2}(x, y) = (x - y)( x^{\prime}(2x^2 - xy - y^2) + y^{\prime}(x^2 + xy - 2y^2) )$$

            If both $W_{2, 3}(a, \overline{x})$ and $W_{2, 2}(a, \overline{x})$ vanish, either $a^{\prime} = \overline{x}^{\prime} = 0$ or the determinant of the system of linear equations $W_{2, 3}(a, \overline{x}) = W_{2, 2}(a, \overline{x}) = 0$ in variables $a^{\prime}$ and $\overline{x}^{\prime}$ vanishes, i.e. $-(a - \overline{x})^5 = 0$. Both cases are impossible since $a \neq \overline{x}$.
    \end{enumerate}

\end{proof}

%%%%%%%%%%%%%%%%%%%%%%%%%%%%%%%%%%%%%%%%%%%%%%%%%%%%%%%%%%%%%%

\begin{lemma}\label{lem:trdeg_estim2}
$\trdeg_k k\langle a, b, \overline{x}, \overline{y}, \overline{t}_1, \ldots, \overline{t}_{n + 2}\rangle \leqslant 2n + 1$.
\end{lemma}

\begin{proof}
Differentiating the equation $\overline{y} - b = \overline{t}_1(a - \overline{x}) + \ldots + \overline{t}_{n + 2}(a^{n + 2} - \overline{x}^{n + 2})$, we obtain the following matrix equality:

    $$
    \begin{pmatrix}
        \overline{y} - b \\
        \overline{y}^{\prime} - b^{\prime} \\
        \vdots \\
        \overline{y}^{(n + 1)} - b^{(n + 1)}
    \end{pmatrix}
    =
    \begin{pmatrix}
        a - \overline{x}                     & \ldots & a^{n + 2} - \overline{x}^{n + 2}                         \\
        a^{\prime} - \overline{x}^{\prime}   & \ldots & \left(a^{n + 2} - \overline{x}^{n + 2} \right)^{\prime}  \\
        \vdots                               & \ddots & \vdots                                                   \\
        a^{(n + 1)} - \overline{x}^{(n + 1)} & \ldots & \left(a^{n + 2} - \overline{x}^{n + 2} \right)^{(n + 1)}
    \end{pmatrix}
    \begin{pmatrix}
        \overline{t}_1      \\
        \overline{t}_2      \\
        \vdots              \\
        \overline{t}_{n + 2}
    \end{pmatrix}
    $$

    Let $k$ be a minimal number such that for all $1 \leqslant l \leqslant k + 1$ the equality $W_{k, l}(a, \overline{x}) = 0$ holds.
    Let us consider two cases:
    \begin{enumerate}
        \item $k < n + 2$. Thus, at least one of $(k - 1) \times (k - 1)$ minors of the matrix:

        $$
        \begin{pmatrix}
            a - \overline{x}                     & \ldots & a^{n + 2} - \overline{x}^{n + 2}                          \\
            \ldots                               & \ddots & \vdots                                                    \\
            a^{(k - 2)} - \overline{x}^{(k - 2)} & \ldots & \left( a^{n + 2} - \overline{x}^{n + 2} \right)^{(k - 2)}
        \end{pmatrix}
        $$

    is nondegenerate. Let $W_{k - 1, l}(a, \overline{x}) \neq 0$. Multiplying by the inverse matrix, we obtain the formulas which express $\overline{t}_j$ as a rational function in $a$, $b$, $\overline{x}$, $\overline{y}$ and their derivations, $\overline{t}_l$, $\overline{t}_{k + 1}, \ldots, \overline{t}_{n + 2}$ for all $1 \leqslant j \leqslant k$ and $j \neq l$.

    Hence, $\overline{t}_1, \ldots, \overline{t}_{l - 1}, \overline{t}_{l + 1}, \ldots, \overline{t}_k \in k\langle a, b, \overline{x}, \overline{y} \rangle(\overline{t}_l, \overline{t}_{k + 1},\ldots, \overline{t}_{n + 2})$. 
    By Lemma \ref{lem:trdeg_estim1}, $\trdeg_k k\langle a, b, \overline{x}, \overline{y}\rangle \leqslant n + k - 2$. 
    Thus

        $$
        \trdeg_k k\langle a, b, \overline{x}, \overline{y}, \overline{t}_1, \ldots, \overline{t}_{n + 2} \rangle \leqslant
        \trdeg_k k\langle a, b, \overline{x}, \overline{y} \rangle + n - k + 3 \leqslant
        (n + k - 2) + (n - k + 3) = 2n + 1
        $$

    \item $k \geqslant n + 2$. 
    In this case, $\trdeg_k k\langle a, b, \overline{x}, \overline{y} \rangle \leqslant 2n$. There exists $1 \leqslant l \leqslant n + 2$ such that $W_{n + 1, l}(a, \overline{x}) \neq 0$. By the same argument as above $\overline{t}_1, \ldots, \overline{t}_{l - 1}, \overline{t}_{l + 1}, \ldots, \overline{t}_{n + 2} \in k\langle a, b, \overline{x}, \overline{y} \rangle(\overline{t}_l)$. The desired inequality is now obvious.  

    \end{enumerate}
\end{proof}

%%%%%%%%%%%%%%%%%%%%%%%%%%%%%%%%%%%%%%%%%%%%%%%%%%%%%%%%%%%%%%%%%%%%%%%%%%%%%%%%%%%%%%%%%%

Lemma \ref{lem:trdeg_estim2} implies that $\overline{t}_1$, $\ldots$, $\overline{t}_{n + 2}$ are algebraically dependent over $k\langle a, b\rangle$.
Let us denote this dependence by $P_j(t_1, \ldots, t_{n + 2}) \in E[t_1, \ldots t_{n + 2}]$.
Consider the polynomial $P = P_{r + 1} \cdot \ldots \cdot P_s$.
Let $\lambda_1, \ldots, \lambda_{n + 2} \in k$ satisfy $P(\lambda_1, \ldots, \lambda_{n + 2}) \neq 0$.
Then, $b - y + \lambda_1(a - x) + \ldots + \lambda_{n + 2}(a^{n + 2} - a^{n + 2}) \neq 0$ for any {generic} solution of $I_0$ other than $(b, a)$. 
Therefore, the proof of the theorem is complete.

\end{proof}

The author is grateful to Dmitry Trushin, Yu.P. Razmyslow, and Lei Fu for useful discussions.


\begin{thebibliography}{4}

    \bibitem{kolchin} Kolchin E.R., \textit{Extensions of differential fields, I}, Annals of Mathematics, vol. 43, 1942.

    \bibitem{bab} Babakhanian A., \textit{On primitive elements in differentially algebraic extension fields}, Trans. AMS, vol. 143, 71-83, 1968.

    \bibitem{ritt} Ritt J.F., \textit{Differential algebra}, Colloquium publications of AMS, vol. 33, 1948.

    %\bibitem{kap} Kaplansky I., \textit{An introduction to differential algebra}, Publications de l'institut de mathematique de l'universite de Nancago, 1957.

    \bibitem{magid} Magid A. R., \textit{Lectures on differential Galois theory}, University lecture series of AMS, vol. 7, 1994.

\end{thebibliography}
\end{document}